\documentclass[11pt]{article}

\usepackage{paralist, amsmath, amsthm, amsfonts, amssymb, xy, color}
\usepackage[margin=1in]{geometry} 
\usepackage{hyperref}

\input xy
\xyoption{all}

\newtheorem{theorem}[equation]{Theorem}
\newtheorem{proposition}[equation]{Proposition}

\newtheorem{corollary}[equation]{Corollary}
\newtheorem{conjecture}[equation]{Conjecture}

\theoremstyle{definition}
\newtheorem{defn}[equation]{Definition}
\newtheorem{eg}[equation]{Example}
\newtheorem{rmk}[equation]{Remark}
\newenvironment{remark}[1][]{\begin{rmk}[#1]\pushQED{\qed}}{\popQED \end{rmk}}

\numberwithin{equation}{section}
\newcommand{\arxiv}[1]{\href{http://arxiv.org/abs/#1}{{\tt arXiv:#1}}}

\newcommand{\Q}{\mathbf{Q}}

\newcommand{\Z}{\mathbf{Z}}

\renewcommand{\phi}{\varphi}

\renewcommand{\tilde}[1]{\widetilde{#1}}

\newcommand{\DS}{\displaystyle}
\newcommand{\rH}{\mathrm{H}}

\makeatletter
\def\Ddots{\mathinner{\mkern1mu\raise\p@
\vbox{\kern7\p@\hbox{.}}\mkern2mu
\raise4\p@\hbox{.}\mkern2mu\raise7\p@\hbox{.}\mkern1mu}}
\makeatother

\renewcommand{\hom}{\operatorname{Hom}} 
\DeclareMathOperator{\rank}{rank} 
\DeclareMathOperator{\gr}{gr}
\newcommand{\GL}{\mathbf{GL}}

\DeclareMathOperator{\End}{End}

\DeclareMathOperator{\Sym}{Sym}

\newcommand{\Sc}{\mathbf{S}}

\newcommand{\rD}{\mathrm{D}}

\newcommand{\bF}{\mathbf{F}}

\newcommand{\bG}{\mathbf{G}}

\newcommand{\bK}{\mathbf{K}}
\newcommand{\cK}{\mathcal{K}}
\newcommand{\bL}{\mathbf{L}}

\newcommand{\cO}{\mathcal{O}}
\newcommand{\cQ}{\mathcal{Q}}
\newcommand{\cR}{\mathcal{R}}
\newcommand{\cS}{\mathcal{S}}
\newcommand{\rS}{\mathrm{S}}
\newcommand{\cT}{\mathcal{T}}
\newcommand{\cU}{\mathcal{U}}

\newcommand{\Gr}{\mathbf{Gr}}

\title{Equations and syzygies of some Kalman varieties}
\author{Steven V Sam}
\date{September 7, 2011}

\begin{document}

\maketitle

\begin{abstract}
  Given a subspace $L$ of a vector space $V$, the Kalman variety
  consists of all matrices of $V$ that have a nonzero eigenvector in
  $L$. Ottaviani and Sturmfels described minimal equations in the case
  that $\dim L = 2$ and conjectured minimal equations for $\dim L =
  3$. We prove their conjecture and describe the minimal free
  resolution in the case that $\dim L = 2$, as well as some related
  results. The main tool is an exact sequence which involves the
  coordinate rings of these Kalman varieties and the normalizations of
  some related varieties. We conjecture that this exact sequence
  exists for all values of $\dim L$.
\end{abstract}

\section*{Introduction.}

Let $V$ be a vector space over a field of arbitrary
characteristic. For a subspace $L \subsetneqq V$, the associated
Kalman variety consists of all matrices that have a nonzero
eigenvector in $L$. A more general definition and basic properties of
Kalman varieties are contained in
Section~\ref{section:kalmandefinition}. The algebraic and geometric
properties of this variety were studied by Ottaviani and Sturmfels in
\cite{sturmfels}, and their definition was motivated by Kalman's
observability condition in control theory \cite{kalman}.

In particular, Ottaviani and Sturmfels find minimal generators for the
prime ideal of the Kalman variety when $\dim L = 2$ and conjecture the
number of equations needed when $\dim L = 3$. (When $\dim L = 1$, the
Kalman variety is an affine space.) Our main results involve
calculating the minimal free resolution in the case $\dim L = 2$
(Theorem~\ref{thm:12n}) and proving their conjecture in the case $\dim
L = 3$ (Theorem~\ref{thm:13neqn}). We point out that even though the
Kalman varieties are of determinantal type in these cases, they are
not Cohen--Macaulay varieties when $\dim V - 1 > \dim L > 1$, so the
resolution is not obtained from the Eagon--Northcott complex.

The main tool is the geometric approach to free resolutions via sheaf
cohomology (Section~\ref{section:geometric}). However, it is not a
straightforward application because this approach only provides
information for the {\it normalization} of the Kalman variety, and the
Kalman variety is not normal whenever $\dim L > 1$. The main insight
into this problem is that the Kalman varieties and their higher
analogues (defined in Section~\ref{section:kalmandefinition}) appear
to have a certain inductive structure. We prove that this structure
exists when $\dim L \le 3$ (Theorem~\ref{thm:inductive}) and
conjecture that it exists in general (see
Conjecture~\ref{conj:inductive}). As further evidence, we sketch a
proof of this conjecture in the case when $\dim V = \dim L + 1$ and
the ground field is of characteristic 0 (see
Section~\ref{sec:conjsketch}).

This inductive structure should provide a means to study the equations
of the Kalman variety when $\dim L > 3$. The validity of
Conjecture~\ref{conj:inductive} would make the Kalman varieties a good
testing ground for studying the equations and free resolutions of
non-normal varieties. In particular, there are very few known
instances where the approach described in
Section~\ref{section:geometric} works effectively for non-normal
varieties. One particularly important instance where the approach in
Section~\ref{section:geometric} is relevant but where the varieties
can fail to be normal are the nilpotent orbits in Lie theory
\cite[Chapter 8]{weyman}, so hopefully the insights gained from
studying the easier case of Kalman varieties will be useful in more
complicated situations.

The outline of the article is as follows. In Section~\ref{sec:prelim},
we summarize the properties of Kalman varieties that we will be using,
as well as the necessary constructions and theorems needed to use the
geometric approach to free resolutions. In Section~\ref{sec:normal},
we prove a few preparatory results on the normalizations of Kalman
varieties, which we use in Section~\ref{sec:equations} to prove our
main results.

\subsection*{Acknowledgements.} 

The author thanks Bernd Sturmfels for showing him \cite[Conjecture
3.6]{sturmfels} which was the motivation for this article. The author
also thanks Giorgio Ottaviani and Bernd Sturmfels for helpful comments
on an earlier draft. The author was supported by an NSF graduate
research fellowship and an NDSEG fellowship while this work was done.

\section{Preliminaries.} \label{sec:prelim}

\subsection{Kalman varieties.} \label{section:kalmandefinition}

Fix a field $K$, a vector space $V$, and a subspace $L \subsetneqq
V$. Set $W = (V/L)^*$, and let $\End(V)$ be the space of linear
operators on $V$ with coordinate ring $A = \Sym(\End(V)^*)$, which is
graded via $\deg \End(V)^* = 1$. Also, let $n = \dim V$, $d = \dim L$
and pick $1 \le s \le d$. The {\bf Kalman variety} is
\[
\cK_{s,d,n} = \{ \phi \in \End(V) \mid \text{there exists } U \subset
L \text{ such that } \dim U = s \text{ and } \phi(U) \subseteq U \}. 
\]
Equations that define $\cK_{s,d,n}$ (at least set-theoretically) can
be obtained as follows. Pick an ordered basis for $V$ starting with a
basis for $L$ followed by a basis for $V/L$ and write $\phi$ in block
matrix form $\begin{pmatrix} \alpha & \beta \\ \gamma &
  \delta \end{pmatrix}$. Then $\cK_{s,d,n}$ is the zero locus of the
$(d-s+1) \times (d-s+1)$ minors of the {\bf reduced Kalman matrix}
\begin{align} \label{eqn:reducedkalman}
\begin{pmatrix} \gamma \\ \gamma\alpha \\ \vdots \\ \gamma
  \alpha^{d-1} \end{pmatrix}
\end{align}
\cite[Theorem 4.5]{sturmfels}. These equations are far from minimal,
and it is unclear if they define a prime ideal. Note that
$\cK_{s,d,n}$ carries an action of the group
\[
P = \{ g \in \GL(V) \mid g(L) = L \},
\]
but often we will just use the symmetry provided by the subgroup $G
\cong \GL(L) \times \GL(W)$ of $P$.

Let $\Gr(s,L)$ denote the Grassmannian of $s$-dimensional subspaces of
$L$. Then $\Gr(s,L)$ has a tautological sequence of vector bundles
\begin{align*} 
0 \to \cR \to L \times \Gr(s,L) \to \cQ \to 0
\end{align*}
where $\rank \cR = s$ and $\rank \cQ = d-s$. Consider the subbundle
$\cS$ of $\End(V) \times \Gr(s,L)$ defined by
\[
\cS = \{(\phi, U) \mid \phi(U) \subseteq U \}.
\]
The bundle $\cS$ is not completely reducible, but it has a filtration
whose associated graded is 
\[
\gr \cS = \End(\cR) \oplus \hom(V/\cR, V).
\]
For later use, set $\xi = ((\End(V) \times \Gr(s,L)) / \cS)^*$. Then
\begin{align*}
\xi = \cR \otimes (\cQ^* \oplus W).
\end{align*}
Let $p_1 \colon \End(V) \times \Gr(s,L) \to \End(V)$ be the
projection. Then $p_1(\cS) = \cK_{s,d,n}$ and $p_1 \colon \cS \to
\cK_{s,d,n}$ is a projective birational morphism.

For $s=d$, $\cK_{d,d,n}$ is isomorphic to affine space and its
defining ideal is generated by $L \otimes W \subset A_1$. For $s<d$,
we can deduce from the map $p_1$ that the singular locus and
non-normal locus of $\cK_{s,d,n}$ coincide and is $\cK_{s+1,d,n}$, and
that $\cK_{s,d,n}$ is an irreducible subvariety of codimension
$s(n-d)$ in $\End(V)$ \cite[Theorem 4.4]{sturmfels}. In particular,
when $n > d+1$, $\cK_{s,d,n}$ is not Cohen--Macaulay by Serre's
criterion for normality. When $s = 1$ and $n = d+1$, $\cK_{1,d,d+1}$
is a hypersurface and hence is Cohen--Macaulay, but we do not know
what happens when $s>1$ and $n=d+1$ in general.

\subsection{Characteristic-free multilinear
  algebra.} \label{sec:charfree} 

Given a partition $\lambda = (\lambda_1, \dots, \lambda_n)$, let
$\ell(\lambda)$ be the largest $i$ such that $\lambda_i \ne 0$. If
$\sum_i \lambda_i = n$, we write $\lambda \vdash n$ and $|\lambda| =
n$. The dual partition $\lambda'$ is defined by $\lambda'_i = \#\{j
\mid \lambda_j \ge i \}$. The notation $a^b$ means the sequence $(a,
\dots, a)$ ($b$ times). Given a partition, we can think of it as a
collection of boxes $(i,j)$ where $1 \le j \le \lambda_i$. The content
of $(i,j)$ is $c(i,j) = j-i$ and the hook length is $h(i,j) =
\lambda_i - j + \lambda'_j - i + 1$.

Let $R$ be a commutative ring and let $U$ be a free $R$-module of
finite rank $n$. We define the determinant of $U$ to be $\det U =
\bigwedge^n U$. The Schur and Weyl functors are denoted $\bL_\lambda
U$ and $\bK_\lambda U$, respectively. See \cite[Chapter 2]{weyman} for
their definition. However, we will change notation from \cite[Chapter
2]{weyman} so that we use $\bL_{\lambda'} U$ to mean $\bL_\lambda
U$. In particular, $\bL_d U \cong \rS^d U$, $\bL_{1^d} U \cong
\bigwedge^d U \cong \bK_{1^d} U$ and $\bK_d U \cong \rD^d U$, where
$\rS$ denotes symmetric powers and $\rD$ denotes divided powers.

We recall the relevant properties that we need. First, both
$\bK_\lambda U$ and $\bL_\lambda U$ are representations of
$\GL(U)$. Both $\bK_\lambda U$ and $\bL_\lambda U$ are free
$U$-modules of the same rank, and this rank is given by
\begin{align} \label{eqn:hookcontent}
\rank \bL_\lambda U = \rank \bK_\lambda U = \prod_{(i,j) \in \lambda}
\frac{ n + c(i,j)}{h(i,j)}
\end{align}
\cite[Corollary 7.21.4]{stanley}. In particular, we have
$\bL_\lambda(U) = 0$ if and only if $\ell(\lambda) > \rank U$ and
similarly with $\bK_\lambda(U)$. Also, $\bL_{\lambda_1 + 1, \dots,
  \lambda_n + 1} U = \det U \otimes \bL_{\lambda_1, \dots, \lambda_n}
U$, and similarly for $\bK$, so we can use this to define
$\bL_\lambda$ and $\bK_\lambda$ when $\lambda$ is a weakly decreasing
sequence of integers which are allowed to be negative. There is a
canonical isomorphism $\bL_\lambda(U^*) = \bK_\lambda(U)^*$
\cite[Proposition 2.1.18]{weyman}. Also there are isomorphisms
$\bL_{\lambda_1, \dots, \lambda_n}(U^*) = \bL_{-\lambda_n, \dots,
  -\lambda_1} U$ and similarly for $\bK$ \cite[Exercise 2.18]{weyman}.

The functors $\bL_\lambda$ and $\bK_\lambda$ are compatible with base
change. Hence it makes sense to construct $\bL_\lambda \cU$ and
$\bK_\lambda \cU$ when $\cU$ is a locally free sheaf on a scheme. When
$R$ is a $\Q$-algebra (or $\Q$-scheme), we have $\bL_\lambda U \cong
\bK_\lambda U$. In this case, we will use the notation $\Sc_\lambda U$
to make it clear that we are dealing with the characteristic 0
situation. In positive characteristic, they need not be isomorphic,
and this is one of the reasons that some of our proofs will only be
valid in characteristic 0.

Given two free modules $U$ and $U'$, the symmetric powers $\rS^d(U
\otimes U')$ have a $\GL(U) \times \GL(U')$-equivariant filtration
whose associated graded is
\[
\gr \rS^d(U \otimes U') = \bigoplus_{\lambda \vdash d} \bL_\lambda U
\otimes \bL_\lambda U'.
\]
Similarly, the exterior powers $\bigwedge^d(U \otimes U')$ have a
$\GL(U) \times \GL(U')$-equivariant filtration whose associated graded
is
\[
\gr \bigwedge^d(U \otimes U') = \bigoplus_{\lambda \vdash d}
\bL_\lambda U \otimes \bK_{\lambda'} U'
\]
\cite[Theorem 2.3.2]{weyman}. These are the {\bf Cauchy
  identities}. Furthermore, given two partitions $\lambda, \mu$, the
tensor product $\bL_\lambda U \otimes \bL_\mu U$ has a filtration
whose associated graded is of the form
\[
\gr \bL_\lambda U \otimes \bL_\mu U = \bigoplus_{\nu \vdash |\lambda|
  + |\mu|} (\bL_\nu U)^{\oplus c^\nu_{\lambda, \mu}}
\]
\cite[Theorem 3.7]{boffi}. When $R$ is a $\Q$-algebra (or
$\Q$-scheme), the above filtrations become direct sum decompositions.

The $c^\nu_{\lambda, \mu}$ are Littlewood--Richardson coefficients
(see \cite[Theorem 2.3.4]{weyman}). We will only need to know these
numbers when $\lambda = (d)$ or $\lambda = (1^d)$, which we now
explain. We say that $\nu$ is obtained from $\mu$ by adding a
horizontal strip of length $d$ if $|\nu| = |\mu| + d$ and we have the
inequalities $\nu_i \ge \mu_i \ge \nu_{i+1}$ for all $i$. In this
case, we have $c^\nu_{(d), \mu} = 1$. Otherwise, we have $c^\nu_{(d),
  \mu} = 0$. For the case $\lambda = (1^d)$, we use the identity
$c^\nu_{(1^d), \mu} = c^{\nu'}_{(d), \mu'}$. Alternatively,
$c^\nu_{(1^d), \mu}$ is nonzero (and equal to 1) if and only if $|\nu|
= |\mu| + d$ and $\mu_i \le \nu_i \le \mu_i + 1$. These are the {\bf
  Pieri rules}.

\subsection{The geometric approach to
  syzygies.} \label{section:geometric}

Fix a field $K$. Let $X$ be a projective variety and let $U$ be a
vector space, and denote the projections $p_1 \colon U \times X \to U$
and $p_2 \colon U \times X \to X$. Let $\cS \subset U \times X$ be a
subbundle with quotient bundle $\cT$ and set $Y = p_1(\cS) \subset
U$. Also, set $\xi = \cT^*$ and let $A = \Sym(U^*)$ be the coordinate
ring of $U$ with grading given by $\deg U^* = 1$. The notation $A(-i)$
denotes the ring $A$ with a grading shift so that it is generated in
degree $i$. For all $i \in \Z$, define graded $A$-modules
\[
\bF_i = \bigoplus_{j \ge 0} \rH^j(X; \bigwedge^{i+j} \xi) \otimes_K
A(-i-j).
\]

\begin{theorem} \label{thm:geometric}
\begin{enumerate}[\rm (a)]
\item \label{item:geometric1} There exist minimal differentials $d_i
  \colon \bF_i \to \bF_{i-1}$ of degree $0$ so that $\bF_\bullet$ is a
  complex of graded free $A$-modules such that
  \[
  \rH_{-i}(\bF_\bullet) = {\rm R}^i (p_1)_* \Sym(\cS^*).
  \]
  In particular, if the higher direct images of $\Sym(\cS^*)$ vanish
  and $p_1$ is birational, then $\bF_\bullet$ is a minimal $A$-free
  resolution of the normalization of $Y$.
\item \label{item:geometric2} Suppose that $\xi$ is a direct sum of
  locally free sheaves $\xi_1 \oplus \xi_2$. For $r,s \ge 0$, define
  \[
  \bF^{\le r,s}_i = \bigoplus_{j \ge 0} \bigoplus_{k=i+j-s}^r \rH^j(X;
  \bigwedge^k \xi_1 \otimes \bigwedge^{i+j-k} \xi_2) \otimes_K
  A(-i-j)
  \]
  with the convention that negative exterior powers are $0$. Then
  $\bF_\bullet^{\le r,s}$ is a subcomplex of $\bF_\bullet$.
\end{enumerate}
\end{theorem}

\begin{proof} \eqref{item:geometric1} is the content of \cite[Theorems
  5.1.2, 5.1.3]{weyman} and \eqref{item:geometric2} follows from the
  proof of \cite[Lemma 5.2.3]{weyman}.
\end{proof}

Now use the notation from Section~\ref{section:kalmandefinition}. We
consider the case of a Grassmannian $X = \Gr(s,L)$ whose points are
the $s$-dimensional subspaces of $L$. The cotangent bundle of $X$ is
$\cR \otimes \cQ^*$.

\begin{theorem}[Kempf vanishing] \label{thm:kempf} Let $\alpha, \beta$
  be two partitions such that $\alpha_{d-s} \ge \beta_1$. Then
  \[
  \rH^j(\Gr(s,L); \bL_\alpha(\cR^*) \otimes \bL_\beta(\cQ^*))
  = \begin{cases} \bL_{(\alpha, \beta)}(L^*) & \text{if } j = 0, \\ 0 &
    \text{if } j > 0 \end{cases}.
  \]
  Furthermore, if $\alpha_{d-s} < \beta_1$, then $\bL_\alpha(\cR^*)
  \otimes \bL_\beta(\cQ^*)$ has no sections.
\end{theorem}

\begin{proof} For the first statement, see \cite[Theorem
  3.1.1]{frobenius}. For the second statement, the sheaf
  $\bL_\alpha(\cR^*) \otimes \bL_\beta(\cQ^*)$ is the pushforward of a
  line bundle on the flag variety, and this line bundle has global
  sections if and only if $\alpha_{d-s} \ge \beta_1$.
\end{proof}

Given a permutation $w$, we define the {\bf length} of $w$ to be
$\ell(w) = \#\{ i < j \mid w(i) > w(j) \}$. Also, define $\rho = (d-1,
d-2, \dots, 1, 0)$. Given a sequence of integers $\alpha$, we define
$w \bullet \alpha = w(\alpha + \rho) - \rho$.

\begin{theorem}[Borel--Weil--Bott] \label{thm:bott} Suppose that the
  characteristic of $K$ is $0$. Let $\alpha$, $\beta$ be two
  partitions and set $\nu = (\alpha, \beta)$. Then exactly one of the
  following two situations occur.
  \begin{enumerate}[\rm 1.]
  \item There exists $w \ne \mathrm{id}$ such that $w \bullet \nu =
    \nu$. Then all cohomology of $\Sc_\alpha \cQ \otimes \Sc_\beta
    \cR$ vanishes.
  \item There is a $($unique$)$ $w$ such that $\eta = w \bullet \nu$
    is a weakly decreasing sequence. Then
    \[
    \rH^{\ell(w)}(\Gr(s,L); \Sc_\alpha \cQ \otimes \Sc_\beta \cR) =
    \Sc_\eta L
    \]
    and all other cohomology vanishes. 
  \end{enumerate}
\end{theorem}

\begin{proof} See \cite[Corollary 4.1.9]{weyman}. \end{proof}

\section{Normalizations of Kalman varieties.} \label{sec:normal}

Let $\cO_{s,d,n}$ denote the coordinate ring of $\cK_{s,d,n}$ and let
$\tilde{\cO}_{s,d,n}$ denote the normalization of $\cO_{s,d,n}$. In
this section we prove some results on $\cO_{s,d,n}$ that will be used
in the main results of this article (Theorem~\ref{thm:12n} and
Theorem~\ref{thm:13neqn}). Some additional results on the
normalizations can be found in Proposition~\ref{prop:s=d-1} and
Proposition~\ref{prop:n=d+1}. Continue the notation of
Section~\ref{section:kalmandefinition}.

\begin{proposition} Over a field of characteristic $0$, the higher
  direct images of $\cS$ vanish for all $s,d,n$. In particular,
  $\tilde{\cO}_{s,d,n}$ has rational singularities and hence is
  Cohen--Macaulay. The higher direct images also vanish in arbitrary
  characteristic in the case $s=1$ and in the case $s=2$, $d=3$. In
  particular, $\tilde{\cO}_{1,d,n}$ and $\tilde{\cO}_{2,3,n}$ are flat
  over $\Z$.
\end{proposition}

Combined with Theorem~\ref{thm:inductive} we conclude that
$\cO_{1,d,n}$ and $\cO_{2,3,n}$ are also flat over $\Z$.

\begin{proof} First suppose that characteristic is 0. By
  Theorem~\ref{thm:geometric}\eqref{item:geometric1}, it is enough to
  show that $\bF_i = 0$ for $i < 0$. The summands of $\bigwedge^q \xi$
  are of the form $\Sc_\lambda \cR \otimes \Sc_\mu \cQ^* \otimes
  \Sc_\nu W$ where $|\lambda| = q$ and $|\mu| \le q$. From the
  description of Borel--Weil--Bott (Theorem~\ref{thm:bott}), it is
  clear that such a sheaf can only have cohomology in degree at most
  $q$, which proves the claim.

  Now suppose that the characteristic is arbitrary. For $s=1$, the
  claim follows from Kempf vanishing (Theorem~\ref{thm:kempf}) since
  $\gr \cS = \cO + \hom(\cQ, V) + \hom(W^*,V)$, so $\Sym(\gr \cS^*)$
  has no higher cohomology, and hence the same is true for
  $\Sym(\cS^*)$. The case of $s=2$ and $d=3$ will be shown in
  Proposition~\ref{prop:23ntilde}.
\end{proof}

\begin{remark} We expect that the higher direct images vanish for all
  $s,d,n$ and in all characteristics, but we are unable to prove this.
\end{remark}

\begin{proposition} \label{prop:1dntilde} $\tilde{\cO}_{1,d,n}$ has
  $($Castelnuovo--Mumford$)$ regularity $d-1$ and the terms of its
  minimal free resolution $\bF_\bullet$ are
\begin{align*}
  {\bf F}_0 &= A \oplus A(-1) \oplus \cdots \oplus A(-d+1)\\
  {\bf F}_i &= \bigoplus_{a=\max(0,i+2d-1-n)}^{d-1}
  \bK_{(i,1^{d-a-1})} L \otimes \bigwedge^{i+d-a-1} W \otimes
  A(-i-d+1) \quad (1 \le i \le n-d).
\end{align*}
\end{proposition}

\begin{proof} Use the notation of Section~\ref{section:geometric}. We
  have
  \[
  \bigwedge^q \xi = \bigoplus_{i=0}^{d-1} \rS^q \cR \otimes \bigwedge^i
  \cQ^* \otimes \bigwedge^{q-i} W
  \]
  with the convention that negative exterior powers are 0. For $0 \le
  q \le d-1$, we have
  \[
  \rH^j(\Gr(1,L); \rS^q \cR \otimes \bigwedge^i \cQ^*) = \begin{cases} K
    & \text{if } q=i=j, \\ 0 & \text{else} \end{cases}
  \]
  \cite[Proposition 5.5]{exterior}. For $d \le q$, we have by Serre
  duality that
  \begin{align*}
    \rH^j(\Gr(1,L); \rS^q \cR \otimes \bigwedge^i \cQ^*) &=
    \rH^{d-1-j}(\Gr(1,L); \rS^{q-d} \cR^* \otimes \bigwedge^i \cQ)^*
    \otimes \det L\\
    &= \rH^{d-1-j}(\Gr(1,L); \rS^{q-d+1} \cR^* \otimes
    \bigwedge^{d-1-i} \cQ^*)^*.
  \end{align*}
  By Kempf vanishing (Theorem~\ref{thm:kempf}), the last term is 0 for
  $j < d-1$. When $j=d-1$, we get
  \begin{align*}
    \rH^0(\Gr(1,L); \rS^{q-d+1} \cR^* \otimes \bigwedge^{d-1-i}
    \cQ^*)^* = \bL_{(q-d+1,1^{d-1-i})}(L^*)^* =
    \bK_{(q-d+1,1^{d-1-i})}L,
  \end{align*}
  and this term contributes to $\bF_{q-d+1}$. The rest follows from
  Section~\ref{section:geometric}.
\end{proof}

\begin{corollary} \label{cor:1dntilde} Let $\bF_\bullet$ be the
  minimal free resolution of $\tilde{\cO}_{1,d,n}$. For $i>1$, the
  only nonzero components in the differential $\bF_i \to \bF_{i-1}$
  are the maps
  \begin{align*}
    \bK_{i,1^{d-a-1}} L \otimes \bigwedge^{i+d-1-a} W \otimes
    A(-i-d+1) \to \begin{array}{c} \DS \bK_{i-1,1^{d-a-1}} L \otimes
      \bigwedge^{i+d-2-a} W \\ \DS \bK_{i-1,1^{d-a}} L \otimes
      \bigwedge^{i+d-1-a} W \end{array} \otimes A(-i-d+2) ,
  \end{align*}
  with the convention that a term on the right is $0$ if it does not
  appear in $\bF_{i-1}$.
\end{corollary}

\begin{proof} Consider the Koszul complex of $\cO_{\cS}$ over the
  total space of $\End(V) \times \Gr(s,L)$. For simplicity, we work
  over $\Gr(s,L)$ by pushing forward along the projection (which is an
  equivalence since $\End(V) \times \Gr(s,L) \to \Gr(s,L)$ is
  affine). The degree $i+d-1$ component of the map above is obtained
  by applying $\rH^{d-1}$ to the map of sheaves
  \[
  \rS^{i+d-1} \cR \otimes \bigwedge^a \cQ^* \otimes
  \bigwedge^{i+d-1-a} W \to \rS^{i+d-2} \cR \otimes \bigwedge^{a-1}
  \cQ^* \otimes \bigwedge^{i+d-2-a} W \otimes \End(V)
  \]
  in this Koszul complex. The equations for $\cO_{\cS}$ are given by
  $\xi = \cR \otimes (\cQ^* \oplus W) \subset \End(V)$. In particular,
  we can restrict our attention to the map
  \begin{align*}
    \rS^{i+d-1} \cR \otimes \bigwedge^a \cQ^* \otimes
    \bigwedge^{i+d-1-a} W \to \begin{array}{c} \DS \rS^{i+d-2} \cR
      \otimes \bigwedge^{a-1} \cQ^* \otimes \bigwedge^{i+d-1-a} W
      \otimes (\cR \oplus \cQ^*) \\
      \DS \rS^{i+d-2} \cR \otimes \bigwedge^a \cQ^* \otimes
      \bigwedge^{i+d-2-a} W \otimes (\cR \oplus W) \end{array}.
  \end{align*}
Using Serre duality, this is the same as
  taking the dual map of applying $\rH^0$ to
  \begin{align*} \begin{array}{c} \DS \rS^{i-1} \cR^* \otimes
      \bigwedge^{a-1} \cQ \otimes \cQ \otimes \bigwedge^{i+d-1-a} W^* \\
      \DS \rS^{i-1} \cR^* \otimes \bigwedge^{a} \cQ \otimes
      \bigwedge^{i+d-2-a} W^* \otimes W^*\end{array} \to \rS^{i-1}
    \cR^* \otimes \bigwedge^{a} \cQ \otimes \bigwedge^{i+d-1-a} W^*
  \end{align*}
  Since the differentials in the Koszul complex are obtained via
  comultiplication, both of these maps are given by exterior
  multiplication. Hence the map on sections is surjective, which
  implies that our desired maps are injective (and hence nonzero).

  That there are no other nonzero maps follows from
  Theorem~\ref{thm:geometric}\eqref{item:geometric2}.
\end{proof}

\begin{proposition} \label{prop:23ntilde} If the characteristic of $K$
  is $0$, then the first few terms of the minimal free resolution
  ${\bf F}_\bullet$ of $\tilde{\cO}_{2,3,n}$ are:
\begin{align*}
  {\bf F}_0 &= A \oplus A(-1) \oplus A(-2)\\
  {\bf F}_1 &= \begin{array}{c} \bigwedge^2 L \otimes \bigwedge^2 W \\
    L \otimes W \end{array} \otimes A(-2) \oplus L \otimes W \otimes
  A(-3)\\
  {\bf F}_2 &= \begin{array}{c} \Sc_{2,1} L \otimes \bigwedge^3 W \\
    \bigwedge^3 L \otimes \Sc_{2,1} W \\ \rS^2 L \otimes \bigwedge^2
    W \end{array} \otimes A(-3) \oplus \begin{array}{c} \bigwedge^2(L
    \otimes W) \\ \bigwedge^3 L \otimes \Sc_{2,1} W \end{array}
  \otimes A(-4)\\
  {\bf F}_3 &= \begin{array}{c} \Sc_{3,1} L \otimes \bigwedge^4 W \\
    \Sc_{2,1,1} L \otimes \Sc_{2,1,1} W \\ \rS^3 L \otimes \bigwedge^3
    W \end{array} \otimes A(-4) \oplus \begin{array}{c} \Sc_{2,1,1} L
    \otimes \Sc_{2,1,1} W \\ \Sc_{2,1,1} L \otimes \Sc_{2,2} W \\ \rS^3
    L \otimes \bigwedge^3 W \\ \Sc_{2,1} L \otimes \Sc_{2,1}
    W \end{array}  \otimes A(-5)
\end{align*}
The ranks of these $\bF_i$ are the same for any field. Furthermore,
the regularity of $\tilde{\cO}_{2,3,n}$ is $2$.
\end{proposition}

\begin{proof} Since $\dim \Gr(2,3) = 2$, the regularity of
  $\tilde{\cO}_{2,3,n}$ is at most 2 by
  Theorem~\ref{thm:geometric}\eqref{item:geometric1}. So the above
  reduces to calculating the cohomology of $\bigwedge^q \xi$ for $0
  \le q \le 5$, which we first do in characteristic 0. This is a
  straightforward, although tedious, application of the Cauchy
  identity, Pieri rule, and Borel--Weil--Bott theorem (all explained
  in Section~\ref{sec:charfree}), which we omit.

  Now assume that the field has characteristic $p>0$. If $p>5$, then
  we may still use Borel--Weil--Bott to calculate the cohomology of
  $\bigwedge^q \xi$ with $q \le 5$ (this reduces to the statement that
  the $n$th symmetric and divided power functors are naturally
  isomorphic when $n!$ is invertible). In the remaining cases $p \in
  \{2,3,5\}$, the cohomology calculation can be reduced to a finite
  calculation with Macaulay 2 \cite{M2}, which we explain. First, we
  have $\xi = \cR \otimes (\cQ^* \oplus W)$. Since we only go up to
  $\bigwedge^5 \xi$, we see that the terms which appear in the Cauchy
  filtration of $\bigwedge^i \xi$ are the same when $\dim W \ge 5$. So
  we only need to consider the case $\dim W = 5$. Finally, we only
  need to calculate $\rH^1$ and $\rH^2$ since we know the Euler
  characteristic. For $\bigwedge^5 \xi$, we only care about
  $\rH^2$. We use the following code:
\begin{verbatim}
A=ZZ/2[z_0,z_1,z_2];
m=matrix{{z_0,z_1,z_2}};
R = sheaf((ker m) ** A^{1});
Q = sheaf(A^{1});
xi = (R ** dual(Q)) ++ (R ++ R ++ R ++ R ++ R);
for i from 1 to 4 do (
     E = exteriorPower(i,xi);
     print (rank HH^1(E), rank HH^2(E)); )
print rank HH^2(exteriorPower(5,xi));
\end{verbatim}
This outputs the answers
\begin{verbatim}
(1, 0)
(45, 1)
(180, 15)
(310, 145)
705
\end{verbatim}
which is the expected answer. Then repeat the above with 2 replaced by
3 and 5.
\end{proof}

\section{Kalman varieties.} \label{sec:equations}

In this section we prove our main results, which include calculating
the minimal free resolution of $\cO_{1,2,n}$ and the equations of
$\cO_{1,3,n}$. During the course of our work, we discovered the
following conjecture.

\begin{conjecture} \label{conj:inductive}
  Fix $d$. For $s=1,\dots,d$, let $B_s =
  \tilde{\cO}_{s,d,n}(-\frac{s(s-1)}{2})$. There is a long exact
  sequence
  \[
  0 \to \cO_{1,d,n} \to B_1 \to B_2 \to \cdots \to B_d \to 0.
  \]
  Furthermore, the ideal of $\cO_{1,d,n}$ has minimal generators in
  degrees $d, d+1, \dots, \frac{d(d+1)}{2}$. The projective dimension
  of $\cO_{1,d,n}$ is $d(n-d)-d+1$ and its regularity is
  $\frac{d(d+1)}{2} - 1$.
\end{conjecture}

The rest of the section will imply that this conjecture holds for $d
\le 3$, so we record the result.

\begin{theorem} \label{thm:inductive} Conjecture~\ref{conj:inductive}
  holds when $d \le 3$. In particular, there are exact sequences
  \[
  0 \to \cO_{1,2,n} \to \tilde{\cO}_{1,2,n} \to \cO_{2,2,n}(-1) \to 0
  \]
  \[
  0 \to \cO_{1,3,n} \to \tilde{\cO}_{1,3,n} \to
  \tilde{\cO}_{2,3,n}(-1) \to \cO_{3,3,n}(-3) \to 0.
  \]
\end{theorem}

For more precise statements about the number of equations, see
Theorem~\ref{thm:12n} and Theorem~\ref{thm:13neqn}.

We expect that the methods used in these cases will extend to any
given value of $d$, but we have been unable to properly organize the
combinatorics in the case of general $d$. However, we are able to
prove Conjecture~\ref{conj:inductive} in the case $n=d+1$ and
$\mathrm{char}\, K = 0$. We provide a brief sketch of this case in
Section~\ref{sec:conjsketch}.

We point out that we were not able to check the conjecture
computationally even for the first nontrivial case $d=4$ and $n=6$. 

\subsection{Syzygies for $d=2$.}

\begin{theorem} \label{thm:12n} The terms of the minimal free
  resolution $\bF_\bullet$ of $\cO_{1,2,n}$ are given by
  \begin{align*} {\bf F}_i &= \det L \otimes \rD^{i-1} L \otimes
    \bigwedge^{i+1} W \otimes A(-i-1)\\
    & \quad \oplus (\bigwedge^{i+1}(L \otimes W)) / (\rD^{i+1} L
    \otimes \bigwedge^{i+1} W) \otimes A(-i-2)  \quad (1 \le i \le  n-3)\\
    {\bf F}_i &= \bigwedge^{i+1}(L \otimes W) \otimes A(-i-2) \quad
    (n-2 \le i \le 2n-5).
\end{align*}
In particular, the projective dimension of $\cO_{1,2,n}$ is $2n-5$ and
it has regularity $2$.
\end{theorem}

\begin{proof} From Proposition~\ref{prop:1dntilde},
  $\tilde{\cO}_{1,2,n}$ has the following presentation:
  \[
  \begin{array}{c} \bigwedge^2 L \otimes \bigwedge^2 W \otimes A(-2)
    \\ L \otimes W \otimes A(-2) \end{array} \to \begin{array}{c} A \\
    A(-1) \end{array} \to \tilde{\cO}_{1,2,n} \to 0.
  \]
  The map $\bigwedge^2 L \otimes \bigwedge^2 W \otimes A(-2) \to
  A(-1)$ is 0. We can either appeal to
  Theorem~\ref{thm:geometric}\eqref{item:geometric2} or use that no
  such $G$-equivariant map exists. Hence the presentation for
  $\tilde{\cO}_{1,2,n} / \cO_{1,2,n}$ must be $L \otimes W \otimes
  A(-2) \to A(-1)$, and we conclude that the quotient is
  $\cO_{2,2,n}(-1)$.

  Let $\tilde{\bF}_\bullet$ be the minimal free resolution of
  $\tilde{\cO}_{1,2,n}$ from Proposition~\ref{prop:1dntilde} and let
  $\bG_\bullet$ be the Koszul complex on $L \otimes W$ resolving
  $\cO_{2,2,n}(-1)$. We can lift the quotient map $\tilde{\cO}_{1,2,n}
  \to \cO_{2,2,n}(-1)$ to get a map of complexes $\tilde{\bF}_\bullet
  \to \bG_\bullet$. The $i$th term of this map is
  \[
  \begin{array}{c} \rD^i L \otimes \bigwedge^i W \otimes A(-i-1) \\
    \bK_{i,1} L \otimes \bigwedge^{i+1} W \otimes A(-i-1) \end{array}
  \to \bigwedge^i(L \otimes W) \otimes A(-i-1). 
  \]
  We claim that the map from $\rD^i L \otimes \bigwedge^i W$ is an
  inclusion and the map from $\bK_{i,1} L \otimes \bigwedge^{i+1} W$
  is 0. By minimality of $\tilde{\bF}_\bullet$, the map $\rD^i L
  \otimes \bigwedge^i W \to \tilde{\bF}_{i-1}$ is injective, and by
  Corollary~\ref{cor:1dntilde}, the map $\bK_{i,1} L \otimes
  \bigwedge^{i+1} W \to \rD^{i-1} L \otimes \bigwedge^{i-1} W \otimes
  A(-i)$ is zero. By induction on $i$, we get the claim.

  Therefore we know exactly what the minimal cancellations in the
  comparison map $\tilde{\bF}_\bullet \to \bG_\bullet$ are, which
  gives the desired resolution $\bF_\bullet$ via a mapping cone. 
\end{proof}

\begin{remark} In the above proof, we know from general principles
  that the comparison maps $\tilde{\bF}_i \to \bG_i$ must be nonzero
  since both $\tilde{\cO}_{1,2,n}$ and $\cO_{2,2,n}$ are
  Cohen--Macaulay (see the proof of \cite[Proposition
  2.3]{posets}). So one can deduce the required cancellations using
  just representation theory (at least in characteristic 0) without
  understanding the differentials.
\end{remark}

\subsection{Equations for $s=1$ and $d=3$.}

In Proposition~\ref{prop:23ntilde}, we do not know how to write down
the $\Z$-forms for the representations of $G$ involved, so we just
switch to the notation $(\lambda; \mu)$ to mean some $\Z$-form of the
module $\Sc_\lambda L \otimes \Sc_\mu W$ and we also write $(-i)$ in
place of $\otimes A(-i)$. 

Let $M$ be the submodule of $\tilde{\cO}_{2,3,n}$ generated by $A
\oplus A(-1)$. We will show that there exist short exact sequences
\[
0 \to \cO_{1,3,n} \to \tilde{\cO}_{1,3,n} \to M(-1) \to 0, \quad \quad
\quad 0 \to M \to \tilde{\cO}_{2,3,n} \to \cO_{3,3,n}(-2) \to 0,
\]
and use a mapping cone to get the equations for $\cO_{1,3,n}$. 

\begin{proposition} \label{prop:M}
The beginning of the minimal $A$-free resolution of $M$ looks like
\[
\begin{array}{c} (2,1;1^3)(-3) \\ (2;1^2)(-3) \\ (1^3;2,1)(-3) \\
  (1^3;2,1)(-4) \\ (1^3;3)(-5) \end{array} \to \begin{array}{c} (1^2;
  1^2)(-2) \\ (1;1)(-2) \end{array} \to \begin{array}{c} A \\
  A(-1) \end{array} \to M \to 0.
\]
Furthermore, the projective dimension of $M$ is $3n - 10$ and the
regularity of $M$ is $3$.
\end{proposition}

\begin{proof} 
  The presentation of $\tilde{\cO}_{2,3,n}$ is
  \[
  \begin{array}{c} (1^2;1^2)(-2) \\ (1;1)(-2) \\ (1;1)(-3) \end{array}
  \to \begin{array}{c} A \\ A(-1) \\ A(-2) \end{array} \to
  \tilde{\cO}_{2,3,n} \to 0.
  \]
  By minimality, the maps from $(1^2;1^2)(-2)$ and $(1;1)(-2)$ to
  $A(-2)$ are 0, so we see that $\tilde{\cO}_{2,3,n} / M \cong
  \cO_{3,3,n}(-2)$. The first few terms of the comparison map of the
  resolutions of $\tilde{\cO}_{2,3,n}$ and $\cO_{3,3,n}(-2)$ is given
  by
  \[
  \xymatrix{ {\begin{array}{c} (3,1;1^4)(-4) \\ (2,1^2;2,1^2)(-4) \\
        (3;1^3)(-4) \\ (2,1^2;2^2)(-5) \\ (2,1^2;2^2)(-5) \\
        (3;1^3)(-5) \\ (2,1;2,1)(-5) \end{array}} \ar[r] \ar[d] &
    {\begin{array}{c} (2,1;1^3)(-3) \\ (2;1^2)(-3) \\ (1^3;2,1)(-3) \\
        (1^3;2,1)(-4) \\ (1^2;2)(-4) \\ (2;1^2)(-4) \end{array}}
    \ar[r] \ar[d] & {\begin{array}{c} (1^2;1^2)(-2) \\ (1;1)(-2) \\
        (1;1)(-3) \end{array} } \ar[r] \ar[d] & {\begin{array}{c} A \\
        A(-1) \\ A(-2) \end{array}} \ar[d] \\
    {\begin{array}{c} (3;1^3)(-5) \\ (2,1;2,1)(-5) \\ (1^3;
        3)(-5) \end{array}} \ar[r] & {\begin{array}{c} (1^2;2)(-4) \\
        (2;1^2)(-4) \end{array}} \ar[r] & (1;1)(-3) \ar[r] & A(-2) }
  \]
  The maps $(1^2;2)(-4) \to (1;1)(-3)$ and $(2;1^2)(-4) \to (1;1)(-3)$
  in the resolution of $\tilde{\cO}_{2,3,n}$ are the Koszul relations
  on the linear equations $(1;1)(-3)$. This implies that the vertical
  maps between the terms of type $(1^2;2)(-4)$, $(2;1^2)(-4)$,
  $(3;1^3)(-5)$, and $(2,1;2,1)(-5)$ are isomorphisms, and the result
  follows by a mapping cone construction.
\end{proof}

\begin{theorem} \label{thm:13neqn}
  The defining equations for $\cK_{1,3,n}$ are
  \[
  (1^3;1^3)(-3) \oplus (1^3;2,1)(-4) \oplus (1^3;2,1)(-5) \oplus
  (1^3;3)(-6)
  \]
  The projective dimension of $\cK_{1,3,n}$ is $3n-11$ and its
  regularity is $5$.
\end{theorem}

\begin{remark} \label{rmk:13neqn} Using \eqref{eqn:hookcontent}, this
  proves \cite[Conjecture 3.6]{sturmfels}, which says that there are
  $\binom{n-3}{3}$ generators in degree 3, $2\binom{n-2}{3}$
  generators in degrees 4 and 5 each, and $\binom{n-1}{3}$ generators
  in degree 6. All of these equations may be interpreted as $3 \times
  3$ minors of the reduced Kalman matrix \eqref{eqn:reducedkalman}. We
  thank Giorgio Ottaviani for bringing this to our attention.
\end{remark}

\begin{proof} The proof is similar to that of
  Theorem~\ref{thm:12n}. The presentation for $\tilde{\cO}_{1,3,n}$ is
  \[
  \begin{array}{c} (1^3;1^3)(-3) \\ (1^2;1^2)(-3) \\
    (1;1)(-3) \end{array} \to \begin{array}{c} A \\ A(-1) \\
    A(-2) \end{array} \to \tilde{\cO}_{\cK_{1,3,n}} \to 0.
  \]
  The map $(1^3;1^3)(-3) \to A(-1) \oplus A(-2)$ is 0 since there are
  no nonzero such $G$-equivariant maps. Also, the maps from
  $(1^2;1^2)(-3)$ and $(1;1)(-3)$ to $A(-1) \oplus A(-2)$ are
  nonzero. If not, then they give generators for the ideal of
  $\cO_{1,3,n}$. In particular, if we pick an ordered basis for $V$
  which first has a basis for $L$ followed by a basis for $W$, then
  these equations correspond to the $2 \times 2$ minors and the $1
  \times 1$ minors of the bottom-left block submatrix, respectively,
  and we can find matrices in $\cK_{1,3,n}$ for which these equations
  do not vanish. 

  Hence from Proposition~\ref{prop:M}, $\tilde{\cO}_{1,3,n} /
  \cO_{1,3,n} \cong M(-1)$. The first few terms of the comparison maps
  between the free resolutions of $\tilde{\cO}_{1,3,n}$ and $M(-1)$
  are
  \[
  \xymatrix{ {\begin{array}{c} (2,1^2;1^4)(-4) \\ (2,1;1^3)(-4) \\
        (2;1^2)(-4)
      \end{array}} \ar[r] \ar[d] & {\begin{array}{c} (1^3;1^3)(-3) \\
        (1^2;1^2)(-3) \\ (1;1)(-3) \end{array}} \ar[r] \ar[d] &
    {\begin{array}{c} A \\ A(-1) \\ A(-2) \end{array}} \ar[d] 
    \\
    {\begin{array}{c} (2,1;1^3)(-4) \\ (2;1^2)(-4) \\ (1^3;2,1)(-4) \\ 
        (1^3;2,1)(-5) \\ (1^3;3)(-6) \end{array}} \ar[r]
    & {\begin{array}{c} (1^2;1^2)(-3) \\ (1;1)(-3) \end{array}}
    \ar[r] & {\begin{array}{c} A(-1) \\ 
        A(-2) \end{array}} 
  }
  \]
  The vertical maps between the terms $(2,1;1^3)(-4)$ and
  $(2;1^2)(-4)$ are isomorphisms. To see this, it is enough to show
  that the maps $(2,1;1^3)(-4) \to (1^2;1^2)(-3)$ and $(2;1^2)(-4) \to
  (1;1)(-3)$ in the resolution of $\tilde{\cO}_{1,3,n}$ are nonzero,
  but this is the content of Corollary~\ref{cor:1dntilde}. Now the
  result follows by a mapping cone construction.
\end{proof}

\subsection{Equations for $s=d-1$.}

In this section, we assume that $K$ has characteristic 0 and find the
equations for $\cO_{d-1,d,n}$. We can also do this in arbitrary
characteristic when $d=3$ since in this case, the next result
is implied by Proposition~\ref{prop:23ntilde}.

\begin{proposition} \label{prop:s=d-1} When $\mathrm{char}\, K = 0$,
  the first few terms of the minimal $A$-free resolution $\bF_\bullet$
  of $\tilde{\cO}_{d-1,d,n}$ are
  \begin{align*}
    \bF_0 &= \bigoplus_{j=0}^{d-1} A(-j)\\
    \bF_1 &= (1^2;1^2)(-2) \oplus \bigoplus_{j=2}^d (1;1)(-j)\\
    \bF_2 &= (1^3;2,1)(-4) \oplus \bigoplus_{j=3}^{d+1} (2;1^2)(-j)
    \oplus \bigoplus_{j=4}^{d+1} (1^2;2)(-j)
  \end{align*}
\end{proposition}

\begin{proof}
  In this case,
  \[
  \bigwedge^q \xi = \bigoplus_{a=0}^{d-1} \bigwedge^a \cR \otimes \rS^a
  \cQ^* \otimes \bigwedge^{q-a} (\cR \otimes W).
  \]
  So we have to calculate the cohomology of sheaves of the form $\rS^a
  \cQ^* \otimes \Sc_\lambda \cR$.

  By Borel--Weil--Bott (Theorem~\ref{thm:bott}), the sheaf $\rS^a \cQ^*
  \otimes \Sc_\lambda \cR$ has cohomology in degree at most
  $\ell(\lambda)$, and such a term appears in $\bigwedge^{|\lambda|}
  \xi$. By Theorem~\ref{thm:geometric}, this term can only contribute
  to $\bF_i$ with $i=0,1,2$ if $|\lambda|-2 \le \ell(\lambda)$. So the
  only possibilities for $\lambda$ with $|\lambda| = q$ are $(1^q)$,
  $(2,1^{q-2})$, $(3,1^{q-3})$, or $(2,2,1^{q-4})$. We will consider
  each of these four cases individually. Recall that $\rho = (d-1,
  d-2, \dots, 1, 0)$.

  Consider a sequence $(-a,1^q)$. Adding $\rho$, we get $(d-1-a, d-1,
  d-2, \dots, d-q, d-q-2, \dots, 1,0)$. So in order to have nonzero
  cohomology, we need $a=q$ where $0 \le a \le d-1$. We get $\rH^0 =
  K$.

  Now consider a sequence $(-a,2,1^{q-2})$. Adding $\rho$, we get
  $(d-1-a, d, d-2, \dots, d-q+1, d-q-1, d-q-2, \dots, 1,0)$. To get
  nonzero cohomology, we need $a=0$ and $q=2$, or $a \ge 1$ and $q =
  a+1$. In the first case, we get $\rH^1 = \bigwedge^2 L$ and in the
  second case, we get $\rH^a = L$. The first case only comes from the
  sheaf $\rS^2 \cR \otimes \bigwedge^2 W$. The second case only comes
  from $\bigwedge^a \cR \otimes \rS^a \cQ^* \otimes \cR \otimes W$. 

  Now consider a sequence $(-a,3,1^{q-3})$. Adding $\rho$, we get
  $(d-1-a, d+1, d-2, d-3, \dots, d-q+2, d-q, \dots, 1,0)$. So we need
  $a \ge 1$ and $q = a+2$. In this case, $\rH^a = \rS^2 L$. This can
  only come from the sheaf $\Sc_{3,1^{a-1}} \cR \otimes \rS^a \cQ^*
  \otimes \bigwedge^2 W \subset \bigwedge^a \cR \otimes \rS^a \cQ^*
  \otimes \bigwedge^2(\cR \otimes W)$.

  Finally consider a sequence $(-a, 2,2,1^{q-4})$. Adding $\rho$, we
  get $(d-1-a, d, d-1, d-3, \dots, d-q+2, d-q, \dots,1, 0)$. So either
  $a=1$, which gives $\rH^2 = \bigwedge^{q-1} L$ or $a \ge 2$ and $q =
  a+2$, which gives $\rH^a = \bigwedge^2 L$. If the first case
  contributes to $\bF_2$, then $q=4$. This comes from the sheaf
  $\Sc_{2,2} \cR \otimes \cQ^* \otimes \Sc_{2,1} W \subset \cR \otimes
  \cQ^* \otimes \bigwedge^3(\cR \otimes W)$. The second case comes
  from the sheaf $\Sc_{2,2,1^{a-2}} \cR \otimes \rS^a \cQ^* \otimes \rS^2
  W \subset \bigwedge^a \cR \otimes \rS^a \cQ^* \otimes \bigwedge^2(\cR
  \otimes W)$.
\end{proof}

\begin{theorem} Assume either that $\mathrm{char}\, K = 0$ or $d=3$.
  Then the equations for $\cO_{d-1,d,n}$ are
  \begin{align*}
    \bigwedge^2 L \otimes \bigwedge^2 W \otimes A(-2),\quad \quad
    \bigwedge^2 L \otimes \rS^2 W \otimes A(-3).
  \end{align*}
\end{theorem}

The interpretation of these equations is just as in
Remark~\ref{rmk:13neqn}.

\begin{proof}
  Using arguments similar to before, the presentation for
  $\tilde{\cO}_{d-1,d,n} / \cO_{d-1,d,n}$ is
  \[
  \bigoplus_{j=2}^d (1;1)(-j) \to \bigoplus_{j=1}^{d-1} A(-j),
  \]
  The maps in the above are of the form $(1;1)(-j-1) \to A(-j)$ for
  $j=1,\dots,d-1$. So in some choice of basis, the cokernel is
  $\bigoplus_{j=1}^{d-1} \cO_{d,d,n}(-j)$, which is resolved by a
  direct sum of Koszul complexes. The next term in the Koszul complex
  is $\bigoplus_{j=3}^{d+1} [(1^2;2)(-j) \oplus (2;1^2)(-j)]$. Let
  $\bF_\bullet$ be the minimal free resolution of
  $\tilde{\cO}_{d-1,d,n}$. Using arguments similar to before, all
  terms in $\bF_2$ of the form $(1^2;2)(-j)$ and $(2;1^2)(-j)$ have a
  nonzero map to $(1;1)(-j+1)$ in $\bF_1$. Hence the maps from these
  terms to the corresponding terms of the Koszul complex of
  $\tilde{\cO}_{d-1,d,n} / \cO_{d-1,d,n}$ are nonzero, and we finish
  the proof by a mapping cone construction.
\end{proof}

\subsection{Conjecture~\ref{conj:inductive} when
  $n=d+1$.} \label{sec:conjsketch}

In this section, we sketch a proof of Conjecture~\ref{conj:inductive}
in the case when $\mathrm{char}\, K = 0$ and $n = d+1$. Since the
details are fairly involved and because this result is not very
substantial, we will just mention the important points, and offer it
as evidence for the validity of Conjecture~\ref{conj:inductive}.

\begin{proposition} \label{prop:n=d+1} When $\mathrm{char}\, K = 0$,
  the terms of the minimal free resolution $\bF_\bullet$ of
  $\tilde{\cO}_{s,d,d+1}$ are
  \[
  \bF_i = \bigoplus_{\lambda \subseteq (s-i) \times (d-s)} \bigwedge^i
  L \otimes \rS^i W \otimes A(-i(d-s+1) - |\lambda|)
  \]
  where $\lambda \subseteq (s-i) \times (d-s)$ means $\ell(\lambda)
  \le s-i$ and $\lambda_1 \le d-s$, and the empty partition is
  allowed.
\end{proposition}

In particular, the generators in $\bF_0$ can be written as
\[
\bigoplus_{\lambda \subseteq s \times (d-s)} \rH^{|\lambda|}(\Gr(s,L);
\Sc_{\lambda} \cR \otimes \Sc_{\lambda'} \cQ^*) \otimes A(-|\lambda|).
\]
So we see that $\tilde{\cO}_{s,d,d+1}$ is generated by
$\bigoplus_{\lambda \subseteq s \times (d-s)} A(-|\lambda|)$. Let
$C_s$ be the submodule of $\tilde{\cO}_{s,d,d+1}$ generated by
$\bigoplus_{\lambda \subseteq (s-1) \times (d-s)} A(-|\lambda|)$ (this
is unambiguous by the above remark). Also define $C_{d+1} = 0$. Note
that $C_1 = \cO_{1,d,d+1}$ and $C_d = \cO_{d,d,d+1}$.

\begin{proposition} Suppose $\mathrm{char}\, K = 0$. For $s=1, \dots,
  d$, there are short exact sequences
  \[
  0 \to C_s \to \tilde{\cO}_{s,d,d+1} \to C_{s+1}(-s) \to 0.
  \]
  Hence Conjecture~\ref{conj:inductive} is true in the case $n=d+1$.
\end{proposition}

\begin{proof}
  We claim that the first $s-1$ terms of the minimal free resolution
  $\bF^s_\bullet$ of $C_s$ are
  \[
  \bF^s_i = \bigoplus_{\lambda \subseteq (s-1) \times (d-s)}
  \bigwedge^i L \otimes \rS^i W \otimes A(-i-|\lambda|) \quad (0 \le i
  \le s-1).
  \]
  This hypothesis is just strong enough to allow one to prove the
  result and the claim by descending induction on $s$. The case $s =
  d$ is clear since $C_d = \cO_{d,d,d+1}$ and is resolved by a Koszul
  complex.
\end{proof}

\bigskip

\small \noindent Steven V Sam, Department of Mathematics,
Massachusetts Institute of Technology, Cambridge, MA 02139 \\
{\tt ssam@math.mit.edu}, \url{http://math.mit.edu/~ssam/}

\end{document}